\documentclass[11pt,fleqn]{article}

\usepackage{amsmath,amssymb,amsthm}
\usepackage[left=1in,right=1in,top=1in,bottom=1in]{geometry}
\usepackage[pdfpagemode=UseNone,pdfstartview=FitH]{hyperref}

\newcommand{\A}{{\cal A}}
\newcommand{\abs}[1]{\left|#1\right|}
\newcommand{\as}{\ins{as}}
\newcommand{\bdry}[1]{\partial #1}
\newcommand{\bgset}[1]{\big\{#1\big\}}
\newcommand{\G}{{\cal G}}
\newcommand{\closure}[1]{\overline{#1}}
\newcommand{\dualp}[3][]{\left(#2,#3\right)_{#1}}
\newcommand{\half}{\frac{1}{2}}
\newcommand{\hquad}{\hspace{0.08in}}
\newcommand{\ins}[1]{\hquad \text{#1} \hquad}
\newcommand{\ip}[3][]{\left(#2,#3\right)_{#1}}
\newcommand{\isom}{\approx}
\newcommand{\jump}[1]{\left[#1\right]}
\newcommand{\M}{{\cal M}}
\newcommand{\N}{\mathbb N}
\newcommand{\norm}[2][]{\left\|#2\right\|_{#1}}
\renewcommand{\O}{\text{O}}
\renewcommand{\o}{\text{o}}
\newcommand{\PS}[1]{$(\text{PS})_{#1}$}
\newcommand{\QED}{\mbox{\qedhere}}
\newcommand{\R}{\mathbb R}
\newcommand{\resp}[1]{\text{(resp. $#1$) }}
\newcommand{\restr}[2]{\left.#1\right|_{#2}}
\newcommand{\seq}[1]{\left(#1\right)}
\newcommand{\set}[1]{\left\{#1\right\}}

\DeclareMathOperator{\divg}{div}

\newenvironment{enumroman}{\begin{enumerate}

}{\end{enumerate}}

\newtheorem{corollary}{Corollary}[section]
\newtheorem{lemma}[corollary]{Lemma}
\newtheorem{proposition}[corollary]{Proposition}
\newtheorem{theorem}[corollary]{Theorem}

\numberwithin{equation}{section}

\title{\bf On Some Elliptic Interface Problems with Nonhomogeneous Jump Conditions\thanks{{\em MSC2010:} Primary 35J66, Secondary 35J20, 35P30
\newline \smallskip \indent\; {\em Key Words and Phrases:} elliptic interface problems, nonhomogeneous jump conditions, nontrivial solutions, Morse theory, critical groups, nonlinear eigenvalue problems, yang index, anisotropic problems, concentration compactness}}
\author{\bf T. Gnana Bhaskar and Kanishka Perera\\
Department of Mathematical Sciences\\
Florida Institute of Technology\\
Melbourne, FL 32901, USA\\
\em gtenali@fit.edu \& kperera@fit.edu}
\date{}

\begin{document}

\maketitle

\begin{abstract}
We obtain nontrivial solutions of some elliptic interface problems with nonhomogeneous jump conditions that arise in localized chemical reactions and nonlinear neutral inclusions. Our proofs in bounded domains use Morse theoretical arguments, in particular, critical group computations. An extension to the whole space is proved using concentration compactness arguments.
\end{abstract}

\section{Introduction}

In this paper we obtain nontrivial solutions of some elliptic interface problems with nonhomogeneous jump conditions. Such problems arise, for example, in localized chemical reactions and in nonlinear neutral inclusions.

Let $\Omega$ be a bounded domain in $\R^N,\, N \ge 2$, let $\Omega_1$ be a $C^1$-subdomain of $\Omega$ such that $\closure{\Omega}_1 \subset \Omega$ and $\Omega_2 = \Omega \setminus \closure{\Omega}_1$ is connected, and let $\Gamma = \bdry{\Omega_1}$. First we consider the interface problem
\begin{equation} \label{1.1}
\left\{\begin{aligned}
\Delta u & = 0 && \text{in } \Omega \setminus \Gamma\\[10pt]
\jump{u} & = 0, && - \jump{\frac{\partial u}{\partial \nu}} = g(x,u) && \text{on } \Gamma\\[10pt]
u & = 0 && \text{on } \bdry{\Omega},
\end{aligned}\right.
\end{equation}
where $\jump{u} = u_2 - u_1,\, u_i = \restr{u}{\Omega_i},\, i = 1,2$, $\jump{\partial u/\partial \nu} = \partial u_2/\partial \nu - \partial u_1/\partial \nu$, $\partial/\partial \nu$ is the exterior normal derivative on $\Gamma$, and $g \in C(\Gamma \times \R)$ satisfies the subcritical growth condition
\begin{equation} \label{1.2}
|g(x,t)| \le C \left(|t|^{q-1} + 1\right) \quad \forall (x,t) \in \Gamma \times \R
\end{equation}
for some $q \in (1,\bar{2})$. Here
\[
\bar{2} = \begin{cases}
2(N-1)/(N-2), & N \ge 3\\[10pt]
\infty, & N = 2
\end{cases}
\]
is the critical trace exponent. This problem arises, for example, in chemical reactions that are localized due to the presence of a catalyst (see Chadam and Yin \cite{MR1258034} and Pan \cite{MR1399066}).

A weak solution of problem \eqref{1.1} is a function $u \in H^1_0(\Omega)$ satisfying
\[
\int_\Omega \nabla u \cdot \nabla v - \int_\Gamma g(x,u)\, v = 0 \quad \forall v \in H^1_0(\Omega),
\]
where $H^1_0(\Omega)$ is the usual Sobolev space with $\norm{u}^2 = \int_\Omega |\nabla u|^2$. As noted by Tian and Ge \cite{MR2465922} and Nieto and O'Regan \cite{MR2474254} in the ODE case, weak solutions of this problem coincide with critical points of the $C^1$-functional
\[
\Phi(u) = \half \int_\Omega |\nabla u|^2 - \int_\Gamma G(x,u), \quad u \in H^1_0(\Omega),
\]
where $G(x,t) = \int_0^t g(x,s)\, ds$.

We assume that the jump $g$ is asymptotically linear near zero and satisfies the Ambrosetti-Rabinowitz superlinearity condition near infinity:
\begin{equation} \label{1.3}
g(x,t) = \mu t + \o(t) \as t \to 0, \text{ uniformly on } \Gamma
\end{equation}
for some $\mu \in \R$, and
\begin{equation} \label{1.4}
0 < \theta\, G(x,t) \le t\, g(x,t), \quad x \in \Gamma,\, |t| \text{ large}
\end{equation}
for some $\theta > 2$. Integrating \eqref{1.4} gives
\begin{equation} \label{1.5}
G(x,t) \ge c\, |t|^\theta - C \quad \forall (x,t) \in \Gamma \times \R
\end{equation}
for some $c > 0$. A model example is
\[
g(x,t) = \mu t + |t|^{q-2}\, t,
\]
where $q \in (2,\bar{2})$.

Assumption \eqref{1.3} leads us to consider the interface eigenvalue problem
\[
\left\{\begin{aligned}
\Delta u & = 0 && \text{in } \Omega \setminus \Gamma\\[10pt]
\jump{u} & = 0, && - \jump{\frac{\partial u}{\partial \nu}} = \mu u && \text{on } \Gamma\\[10pt]
u & = 0 && \text{on } \bdry{\Omega}.
\end{aligned}\right.
\]
The spectrum $\sigma(\Gamma)$ of this problem consists of an increasing and unbounded sequence of positive eigenvalues of finite multiplicities.

\begin{theorem} \label{Theorem 1.1}
Assume \eqref{1.2} with $q \in (1,\bar{2})$, \eqref{1.3}, and \eqref{1.4} with $\theta > 2$. If $\mu \notin \sigma(\Gamma)$, then problem \eqref{1.1} has a nontrivial solution.
\end{theorem}

Our proof of this theorem will be based on explicitly computing the critical groups of $\Phi$ at zero and at infinity, and showing that they are nonisomorphic in some dimension. These groups are defined by
\[
C_q(\Phi,0) = H_q(\Phi^0 \cap U,\Phi^0 \cap U \setminus \set{0}), \quad C_q(\Phi,\infty) = H_q(H^1_0(\Omega),\Phi^a), \quad q \ge 0,
\]
where $\Phi^0 = \bgset{u \in H^1_0(\Omega) : \Phi(u) \le 0}$, $U$ is any neighborhood of zero, $a < 0$ is such that $\Phi$ has no critical points in $\Phi^a = \bgset{u \in H^1_0(\Omega) : \Phi(u) \le a}$, and $H_\ast(\cdot,\cdot)$ are the relative singular homology groups. We refer the reader to Chang \cite{MR94e:58023} and Mawhin and Willem \cite{MR90e:58016} for the necessary background on Morse theory.

We also indicate how Theorem \ref{Theorem 1.1} can be extended to the $p$-Laplacian case
\begin{equation} \label{1.6}
\left\{\begin{aligned}
\Delta_p\, u & = 0 && \text{in } \Omega \setminus \Gamma\\[10pt]
\jump{u} & = 0, && - \jump{|\nabla u|^{p-2}\, \frac{\partial u}{\partial \nu}} = g(x,u) && \text{on } \Gamma\\[10pt]
u & = 0 && \text{on } \bdry{\Omega},
\end{aligned}\right.
\end{equation}
where $\Delta_p\, u = \divg (|\nabla u|^{p-2}\, \nabla u)$ is the $p$-Laplacian of $u,\, p > 2$, $g$ satisfies \eqref{1.2} with $q \in (1,\bar{p})$, and $\bar{p} = (N-1)p/(N-p)$ if $N > p$ and $\bar{p} = \infty$ if $N \le p$. A weak solution $u \in W^{1,\, p}_0(\Omega)$ of this problem satisfies
\[
\int_\Omega |\nabla u|^{p-2}\, \nabla u \cdot \nabla v - \int_\Gamma g(x,u)\, v = 0 \quad \forall v \in W^{1,\, p}_0(\Omega),
\]
and the associated variational functional is
\[
\Phi(u) = \frac{1}{p} \int_\Omega |\nabla u|^p - \int_\Gamma G(x,u), \quad u \in W^{1,\, p}_0(\Omega).
\]

We replace \eqref{1.3} with
\begin{equation} \label{1.7}
g(x,t) = \mu\, |t|^{p-2}\, t + \o(|t|^{p-1}) \as t \to 0, \text{ uniformly on } \Gamma
\end{equation}
and take $\theta > p$ in \eqref{1.4}. Our model example is $g(x,t) = \mu\, |t|^{p-2}\, t + |t|^{q-2}\, t$, where $q \in (p,\bar{p})$. The corresponding eigenvalue problem is
\begin{equation} \label{1.8}
\left\{\begin{aligned}
\Delta_p\, u & = 0 && \text{in } \Omega \setminus \Gamma\\[10pt]
\jump{u} & = 0, && - \jump{|\nabla u|^{p-2}\, \frac{\partial u}{\partial \nu}} = \mu\, |u|^{p-2}\, u && \text{on } \Gamma\\[10pt]
u & = 0 && \text{on } \bdry{\Omega},
\end{aligned}\right.
\end{equation}
and the spectrum $\sigma_p(\Gamma)$ of this problem consists of the set of all $\mu \in \R$ for which it has a nontrivial solution.

\begin{theorem} \label{Theorem 1.2}
Assume \eqref{1.2} with $q \in (1,\bar{p})$, \eqref{1.4} with $\theta > p$, and \eqref{1.7}. If $\mu \notin \sigma_p(\Gamma)$, then problem \eqref{1.6} has a nontrivial solution.
\end{theorem}

Our proof here will use an increasing and unbounded sequence of eigenvalues constructed via the Yang index as in Perera \cite{MR1998432} (see also Perera et al. \cite{MR2640827}). The standard sequence of eigenvalues based on the Krasnosel$'$ski\u\i\ genus does not provide sufficient information about the critical groups to prove this theorem.

Next we consider the anisotropic problem
\begin{equation} \label{1.9}
\left\{\begin{aligned}
\Delta_p\, u_1 & = 0 && \text{in } \Omega_1\\[10pt]
\Delta u_2 & = 0 && \text{in } \Omega_2\\[10pt]
u_1 & = u_2, && - \left(\frac{\partial u_2}{\partial \nu} - |\nabla u_1|^{p-2}\, \frac{\partial u_1}{\partial \nu}\right) = g(x,u_1) && \text{on } \Gamma\\[10pt]
u_2 & = 0 && \text{on } \bdry{\Omega},
\end{aligned}\right.
\end{equation}
where $p > 2$ and $g$ satisfies \eqref{1.2} with $q \in (1,\bar{p})$. Problems of this type arise in nonlinear neutral inclusions, in particular, in assemblages of neutral coated spheres (see Jim\'{e}nez et al. \cite{JiVeSa}).

We work in the space
\[
X = \bgset{u = (u_1,u_2) \in W^{1,\, p}(\Omega_1) \times H^1(\Omega_2) : u_1 = u_2 \text{ on } \Gamma,\, u_2 = 0 \text{ on } \bdry{\Omega}},
\]
which is a closed subspace of $W^{1,\, p}(\Omega_1) \times H^1(\Omega_2)$. A norm on $X$, equivalent to the standard norm of $W^{1,\, p}(\Omega_1) \times H^1(\Omega_2)$, can be defined by
\[
\norm{u} = \left(\int_{\Omega_1} |\nabla u_1|^p\right)^{1/p} + \left(\int_{\Omega_2} |\nabla u_2|^2\right)^{1/2}.
\]
A weak solution $u \in X$ of problem \eqref{1.9} satisfies
\[
\int_{\Omega_1} |\nabla u_1|^{p-2}\, \nabla u_1 \cdot \nabla v_1 + \int_{\Omega_2} \nabla u_2 \cdot \nabla v_2 - \int_\Gamma g(x,u_1)\, v_1 = 0 \quad \forall v \in X,
\]
and they coincide with critical points of the functional
\[
\Phi(u) = \frac{1}{p} \int_{\Omega_1} |\nabla u_1|^p + \half \int_{\Omega_2} |\nabla u_2|^2 - \int_\Gamma G(x,u_1), \quad u \in X.
\]

\begin{theorem} \label{Theorem 1.3}
Assume \eqref{1.2} with $q \in (1,\bar{p})$,
\begin{equation} \label{1.10}
g(x,t) = \o(|t|^{p-1}) \as t \to 0, \text{ uniformly on } \Gamma,
\end{equation}
and \eqref{1.4} with $\theta > p$. Then problem \eqref{1.9} has a nontrivial solution.
\end{theorem}

Finally we consider an interface eigenvalue problem in the whole space $\R^N,\, N \ge 2$, with \linebreak $\R^{N-1} \times \set{0}$ as the interface, which we identify with $\R^{N-1}$. We consider
\begin{equation} \label{1.11}
\left\{\begin{aligned}
- \Delta u + V(x)\, u & = 0 && \text{in } \R^N \setminus \R^{N-1}\\[10pt]
\jump{u} & = 0, && - \jump{\frac{\partial u}{\partial x_N}} = \lambda\, |u|^{p-2}\, u && \text{on } \R^{N-1}\\[10pt]
u(x) & \to 0 && \text{as } |x| \to \infty,
\end{aligned}\right.
\end{equation}
where $V \in L^\infty(\R^N)$, $\jump{u} = u_+ - u_-,\, u_+ = \restr{u}{\R^{N-1} \times (0,+ \infty)},\, u_- = \restr{u}{\R^{N-1} \times (- \infty,0)}$, $\jump{\partial u/\partial x_N} = \partial u_+/\partial x_N - \partial u_-/\partial x_N$, $\lambda \in \R$, and $p \in (2,\bar{2})$. Writing $\R^N = \R^{N-1} \oplus \R,\, x = (x',x_N)$, we assume that
\begin{equation} \label{1.12}
\inf_{x \in \R^N} V(x) > 0, \qquad \lim_{|x'| \to \infty} V(x) = V^\infty > 0 \quad \forall x_N \in \R.
\end{equation}
A weak solution of this problem is a function $u \in H^1(\R^N)$ satisfying
\[
\int_{\R^N} \nabla u \cdot \nabla v + V(x)\, uv = \lambda \int_{\R^{N-1}} |u|^{p-2}\, uv \quad \forall v \in H^1(\R^N),
\]
where $H^1(\R^N)$ is the usual Sobolev space with the norm $\norm{\cdot}$ induced by the inner product
\[
\ip{u}{v} = \int_{\R^N} \nabla u \cdot \nabla v + V^\infty\, uv,
\]
which is equivalent to the standard norm.

Let
\[
I(u) = \int_{\R^{N-1}} |u|^p, \quad J(u) = \int_{\R^N} |\nabla u|^2 + V(x)\, u^2, \quad u \in H^1(\R^N).
\]
Then the eigenfunctions of \eqref{1.11} on the manifold
\[
\M = \bgset{u \in H^1(\R^N) : I(u) = 1}
\]
and the corresponding eigenvalues coincide with the critical points and the critical values of the constrained functional $\restr{J}{\M}$, respectively.

We will first show that the autonomous problem at infinity,
\[
\left\{\begin{aligned}
- \Delta u + V^\infty\, u & = 0 && \text{in } \R^N \setminus \R^{N-1}\\[10pt]
\jump{u} & = 0, && - \jump{\frac{\partial u}{\partial x_N}} = \lambda\, |u|^{p-2}\, u && \text{on } \R^{N-1}\\[10pt]
u(x) & \to 0 && \text{as } |x| \to \infty,
\end{aligned}\right.
\]
has a least energy solution on $\M$. Let
\[
J^\infty(u) = \int_{\R^N} |\nabla u|^2 + V^\infty\, u^2, \quad u \in H^1(\R^N)
\]
be the corresponding functional. Then
\[
\lambda_1^\infty := \inf_{u \in \M}\, J^\infty(u) > 0
\]
by the Sobolev trace imbedding.

\begin{theorem} \label{Theorem 1.4}
Assume $V^\infty > 0$ and $p \in (2,\bar{2})$. Then the infimum $\lambda_1^\infty$ is attained at a function $w_1^\infty > 0$.
\end{theorem}

For the nonautonomous problem, $\sqrt{J(\cdot)}$ is an equivalent norm on $H^1(\R^N)$ since $V \in L^\infty(\R^N)$ and $\inf V > 0$, so
\[
\lambda_1 := \inf_{u \in \M}\, J(u) > 0.
\]
By the invariance of $J^\infty$ with respect to the group $D$ of $(N-1)$-dimensional shifts $u \mapsto u(\cdot - y), \linebreak y \in \R^{N-1}$, we have $\lambda_1 \le \lambda_1^\infty$, and we will show that $\lambda_1$ is attained if the inequality is strict.

\begin{theorem} \label{Theorem 1.5}
Assume \eqref{1.12} and $p \in (2,\bar{2})$. If $\lambda_1 < \lambda_1^\infty$, then $\lambda_1$ is attained at a function $w_1 > 0$.
\end{theorem}

The inequality $\lambda_1 < \lambda_1^\infty$ holds if $V \le V^\infty$, with the strict inequality on a set of positive measure. Indeed, in that case
\[
\lambda_1 \le J(w_1^\infty) < J^\infty(w_1^\infty) = \lambda_1^\infty
\]
since $w_1^\infty > 0$ a.e. So we have the following corollary.

\begin{corollary}
Assume \eqref{1.12} and $p \in (2,\bar{2})$. If $V(x) \le V^\infty$ for all $x \in \R^N$ and the strict inequality holds on a set of positive measure, then $\lambda_1$ is attained at a function $w_1 > 0$.
\end{corollary}

The main difficulty here is the lack of compactness inherent in this problem. This lack of compactness originates from the invariance of $\R^N$ and $\R^{N-1}$ under the action of the noncompact group $D$, and manifests itself in the noncompactness of the Sobolev trace imbedding \linebreak $H^1(\R^N) \hookrightarrow L^p(\R^{N-1})$, which in turn implies that the manifold $\M$ is not weakly closed in $H^1(\R^N)$. Our proofs will use the concentration compactness principle of Lions \cite{MR778970,MR879032}, expressed as a suitable profile decomposition for minimizing sequences, to overcome this difficulty.

\section{Proofs of Theorems \ref{Theorem 1.1} -- \ref{Theorem 1.3}}

\subsection{Proof of Theorem \ref{Theorem 1.1}}

Consider the linear interface problem
\[
\left\{\begin{aligned}
\Delta u & = 0 && \text{in } \Omega \setminus \Gamma\\[10pt]
\jump{u} & = 0, && - \jump{\frac{\partial u}{\partial \nu}} = f(x) && \text{on } \Gamma\\[10pt]
u & = 0 && \text{on } \bdry{\Omega},
\end{aligned}\right.
\]
where $f \in L^2(\Gamma)$. This problem has a solution $u$, obtained by minimizing the associated functional
\[
\Psi(u) = \half \int_\Omega |\nabla u|^2 - \int_\Gamma f(x)\, u, \quad u \in H^1_0(\Omega),
\]
and it is unique by the maximum principle. Testing $\Psi'(u) = 0$ with $u$ and using the H\"{o}lder inequality gives $\norm{u}^2 \le \norm[L^2(\Gamma)]{f} \norm[L^2(\Gamma)]{u}$. So the linear map $L^2(\Gamma) \to H^1_0(\Omega),\, f \mapsto u$ is bounded by the boundedness of the trace imbedding $H^1_0(\Omega) \hookrightarrow L^2(\Gamma)$, and hence the solution operator $S : L^2(\Gamma) \to L^2(\Gamma),\, Sf = u$ is compact by the compactness of the same imbedding. Thus, the spectrum $\sigma(\Gamma)$ of the self-adjoint operator $S^{-1}$ consists of isolated eigenvalues $\mu_l \nearrow \infty$, of multiplicities $d_l < \infty$.

Since $\mu \notin \sigma(\Gamma)$, zero is a nondegenerate critical point of the asymptotic functional
\[
\Phi_0(u) = \half \int_\Omega |\nabla u|^2 - \frac{\mu}{2} \int_\Gamma u^2, \quad u \in H^1_0(\Omega),
\]
of Morse index
\[
m_0 = \sum_{\mu_l < \mu}\, d_l.
\]
Then it follows from a standard homotopy argument that
\[
C_q(\Phi,0) \isom C_q(\Phi_0,0) = \delta_{q m_0}\, \G,
\]
where $\G$ is the coefficient group.

Recall that $\Phi$ satisfies the Palais-Smale compactness condition \PS{} if every sequence $\seq{u_j}$ in $H^1_0(\Omega)$ such that $\Phi(u_j)$ is bounded and $\Phi'(u_j) \to 0$, called a \PS{} sequence for $\Phi$, has a convergent subsequence.

\begin{lemma} \label{Lemma 2.1}
If \eqref{1.2} and \eqref{1.4} hold, then
\begin{enumroman}
\item \label{Lemma 2.1.i} $\Phi$ satisfies the {\em \PS{}} condition,
\item \label{Lemma 2.1.ii} $C_q(\Phi,\infty) = 0$ for all $q$.
\end{enumroman}
\end{lemma}

\begin{proof}
\ref{Lemma 2.1.i} Let $\seq{u_j}$ be a \PS{} sequence for $\Phi$. We have
\[
\left(\frac{\theta}{2} - 1\right)\! \norm{u_j}^2 = \theta\, \Phi(u_j) - \dualp{\Phi'(u_j)}{u_j} + \int_\Gamma \theta\, G(x,u_j) - u_j\, g(x,u_j) \le \o(\norm{u_j}) + \O(1)
\]
by \eqref{1.2} and \eqref{1.4}, so $\norm{u_j}$ is bounded and then a standard argument gives a convergent subsequence.

\ref{Lemma 2.1.ii} We have the orthogonal decomposition $H^1_0(\Omega) = V \oplus W,\, u = v + w$, where $V = \linebreak H^1_0(\Omega_1) \oplus H^1_0(\Omega_2)$ and $W = V^\perp$. For $v \in V$,
\[
\Phi(v) = \half \int_\Omega |\nabla v|^2 \ge 0.
\]
Denoting by $S = \bgset{u \in H^1_0(\Omega) : \norm{u} = 1}$ the unit sphere in $H^1_0(\Omega)$, for $u \in S \setminus V$ and $\tau > 0$,
\begin{equation} \label{2.1}
\Phi(\tau u) = \frac{\tau^2}{2} - \int_\Gamma G(x,\tau w) \le \frac{\tau^2}{2} - c\, \tau^\theta \int_\Gamma |w|^\theta + C \to - \infty \as \tau \to \infty
\end{equation}
by \eqref{1.5}, and
\begin{equation} \label{2.2}
\frac{d}{d \tau}\, \big(\Phi(\tau u)\big) = \tau - \int_\Gamma w\, g(x,\tau w) = \frac{2}{\tau} \left(\Phi(\tau u) - \int_\Gamma \frac{\tau w}{2}\, g(x,\tau w) - G(x,\tau w)\right) \le \frac{2}{\tau}\, \big(\Phi(\tau u) - a_0\big)
\end{equation}
for some $a_0 \le 0$ by \eqref{1.2} and \eqref{1.4}. Fix $a < a_0$. Then $\Phi$ has no critical points in $\Phi^a$.

We have $\Phi(\tau u) \le a$ for all sufficiently large $\tau$ by \eqref{2.1}, and
\[
\Phi(\tau u) \le a \implies \frac{d}{d \tau}\, \big(\Phi(\tau u)\big) < 0
\]
by \eqref{2.2}, so there is a unique $\tau_u > 0$ such that
\[
\tau < \resp{=, >} \tau_u \implies \Phi(\tau u) > \resp{=, <} a
\]
and the map $S \setminus V \to (0,\infty),\, u \mapsto \tau_u$ is $C^1$ by the implicit function theorem. Then
\[
\Phi^a = \bgset{\tau u : u \in S \setminus V,\, \tau \ge \tau_u},
\]
and $H^1_0(\Omega) \setminus V$ radially deformation retracts to $\Phi^a$ via
\[
(H^1_0(\Omega) \setminus V) \times [0,1] \to H^1_0(\Omega) \setminus V, \quad (u,t) \mapsto \begin{cases}
(1 - t)\, u + t\, \tau_{\widehat{u}}\, \widehat{u}, & u \in (H^1_0(\Omega) \setminus V) \setminus \Phi^a\\[10pt]
u, & u \in \Phi^a,
\end{cases}
\]
where $\widehat{u} = u/\norm{u}$ is the projection of $u \ne 0$ on $S$.

Thus,
\[
C_q(\Phi,\infty) = H_q(H^1_0(\Omega),\Phi^a) \isom H_q(H^1_0(\Omega),H^1_0(\Omega) \setminus V) \isom H_q(W,W \setminus \set{0}) = 0 \quad \forall q
\]
since $W$ is infinite dimensional.
\end{proof}

Theorem \ref{Theorem 1.1} now follows since $C_{m_0}(\Phi,0) \not\isom C_{m_0}(\Phi,\infty)$.

\subsection{Proof of Theorem \ref{Theorem 1.2}}

Let
\[
I(u) = \int_\Omega |\nabla u|^p, \quad J(u) = \int_\Gamma |u|^p, \quad u \in W^{1,\, p}_0(\Omega).
\]
Then the eigenvalues of problem \eqref{1.8} coincide with the critical values of the functional $I$ on the manifold
\[
\M = \bgset{u \in W^{1,\, p}_0(\Omega) : J(u) = 1}
\]
by the Lagrange multiplier rule. Denote by $\A$ the class of closed symmetric subsets of $\M$ and by $i(A)$ the Yang index of $A \in \A$ (see Yang \cite{MR16:502d,MR17:289e}). Then
\[
\mu_l := \inf_{\substack{A \in \A\\[1pt]
i(A) \ge l-1}} \sup_{u \in A}\, I(u), \quad l \ge 1
\]
is an increasing and unbounded sequence of eigenvalues. Moreover, zero is an isolated critical point of
\[
\Phi_0(u) = \frac{1}{p} \int_\Omega |\nabla u|^p - \frac{\mu}{p} \int_\Gamma |u|^p, \quad u \in W^{1,\, p}_0(\Omega)
\]
since $\mu \notin \sigma_p(\Gamma)$, and $C_q(\Phi_0,0) = \delta_{q0}\, \G$ if $\mu < \mu_1$ and $C_l(\Phi_0,0) \ne 0$ if $\mu_l < \mu < \mu_{l+1}$ (see Perera \cite{MR1998432} and Perera et al. \cite{MR2640827}). Since $C_q(\Phi,0) \isom C_q(\Phi_0,0)$, Theorem \ref{Theorem 1.2} then follows as before once we prove the following lemma.

\begin{lemma}
If \eqref{1.2} with $q \in (1,\bar{p})$ and \eqref{1.4} with $\theta > p$ hold, then
\begin{enumroman}
\item \label{Lemma 2.2.i} $\Phi$ satisfies the {\em \PS{}} condition,
\item \label{Lemma 2.2.ii} $C_q(\Phi,\infty) = 0$ for all $q$.
\end{enumroman}
\end{lemma}

\begin{proof}
\ref{Lemma 2.2.i} If $\seq{u_j}$ is a \PS{} sequence for $\Phi$, then
\[
\left(\frac{\theta}{p} - 1\right)\! \norm{u_j}^p = \theta\, \Phi(u_j) - \dualp{\Phi'(u_j)}{u_j} + \int_\Gamma \theta\, G(x,u_j) - u_j\, g(x,u_j) \le \o(\norm{u_j}) + \O(1)
\]
by \eqref{1.2} and \eqref{1.4}, so $\norm{u_j}$ is bounded and a standard argument gives a convergent subsequence.

\ref{Lemma 2.2.ii} The argument in the proof of Lemma \ref{Lemma 2.1} \ref{Lemma 2.1.ii} can be repeated with the direct sum decomposition $W^{1,\, p}_0(\Omega) = V \oplus W$, where $V = W^{1,\, p}_0(\Omega_1) \oplus W^{1,\, p}_0(\Omega_2)$ and $W = \bgset{u \in W^{1,\, p}_0(\Omega) : \Delta u = 0 \text{ in } \Omega \setminus \Gamma}$, i.e., $u \in W^{1,\, p}_0(\Omega)$ belongs to $W$ if and only if
\[
\int_\Omega \nabla u \cdot \nabla v = 0 \quad \forall v \in V. \QED
\]
\end{proof}

\subsection{Proof of Theorem \ref{Theorem 1.3}}

By the elementary inequality $(a + b)^p \le 2^p (a^p + b^2),\, a \ge 0,\, 0 \le b \le 1$, \eqref{1.2}, and \eqref{1.10},
\[
\Phi(u) \ge \left(\frac{1}{2^p\, p} + \o(1)\right)\! \norm{u}^p \as \norm{u} \to 0,
\]
so zero is a strict local minimizer of $\Phi$ and hence $C_q(\Phi,0) = \delta_{q0}\, \G$. Theorem \ref{Theorem 1.3} now follows from the following lemma as before.

\begin{lemma}
If \eqref{1.2} with $q \in (1,\bar{p})$ and \eqref{1.4} with $\theta > p$ hold, then
\begin{enumroman}
\item \label{Lemma 2.3.i} $\Phi$ satisfies the {\em \PS{}} condition,
\item \label{Lemma 2.3.ii} $C_q(\Phi,\infty) = 0$ for all $q$.
\end{enumroman}
\end{lemma}

\begin{proof}
\ref{Lemma 2.3.i} If $\seq{u_j}$ is a \PS{} sequence for $\Phi$, then
\begin{multline*}
\left(\frac{\theta}{p} - 1\right) \int_{\Omega_1} |\nabla u_{j1}|^p + \left(\frac{\theta}{2} - 1\right) \int_{\Omega_2} |\nabla u_{j2}|^2 = \theta\, \Phi(u_j) - \dualp{\Phi'(u_j)}{u_j}\\[10pt]
+ \int_\Gamma \theta\, G(x,u_{j1}) - u_{j1}\, g(x,u_{j1}) \le \o(\norm{u_j}) + \O(1)
\end{multline*}
by \eqref{1.2} and \eqref{1.4}, so $\norm{u_j}$ is bounded and a standard argument gives a convergent subsequence.

\ref{Lemma 2.3.ii} The argument in the proof of Lemma \ref{Lemma 2.1} \ref{Lemma 2.1.ii} can be repeated with the direct sum decomposition $X = V \oplus W$, where $V = W^{1,\, p}_0(\Omega_1) \oplus H^1_0(\Omega_2)$ and $W = \bgset{u \in X : \Delta u = 0 \text{ in } \Omega \setminus \Gamma}$, i.e., $u \in X$ belongs to $W$ if and only if
\[
\int_{\Omega_1} \nabla u_1 \cdot \nabla v_1 + \int_{\Omega_2} \nabla u_2 \cdot \nabla v_2 = 0 \quad \forall v \in V. \QED
\]
\end{proof}

\section{Proofs of Theorems \ref{Theorem 1.4} and \ref{Theorem 1.5}}

In the absence of a compact Sobolev trace imbedding, the main technical tool we use here for handling the convergence matters is the concentration compactness principle of Lions \cite{MR778970,MR879032}. This is expressed as the following profile decomposition for bounded sequences in $H^1(\R^N)$ (see Tintarev and Fieseler \cite{MR2294665}).

\begin{proposition} \label{Proposition 3.1}
Let $u_k \in H^1(\R^N)$ be a bounded sequence, and assume that there is a constant $\delta > 0$ such that if $u_k(\cdot + y_k) \rightharpoonup w \ne 0$ on a renumbered subsequence for some $y_k \in \R^{N-1}$ with $|y_k| \to \infty$, then $\norm{w} \ge \delta$. Then there are $m \in \N$, $w^{(n)} \in H^1(\R^N)$, $y^{(n)}_k \in \R^{N-1},\, y^{(1)}_k = 0$ with $k \in \N$, $n \in \set{1,\dots,m}$, $w^{(n)} \ne 0$ for $n \ge 2$, such that, on a renumbered subsequence,
\begin{gather}
u_k(\cdot + y^{(n)}_k) \rightharpoonup w^{(n)}, \label{3.1}\\[10pt]
\big|y^{(n)}_k - y^{(l)}_k\big| \to \infty \text{ for } n \ne l, \label{3.2}\\[7.5pt]
\sum_{n=1}^m\, \norm{w^{(n)}}^2 \le \liminf\, \norm{u_k}^2, \label{3.3}\\[5pt]
u_k - \sum_{n=1}^m\, w^{(n)}(\cdot - y^{(n)}_k) \to 0 \text{ in } L^p(\R^{N-1}) \quad \forall p \in (2,\bar{2}). \label{3.4}
\end{gather}
\end{proposition}


Recall that $u_k \in \M$ is a critical sequence for $\restr{J}{\M}$ at the level $c \in \R$ if
\begin{equation} \label{3.5}
J'(u_k) - \mu_k\, I'(u_k) \to 0, \qquad J(u_k) \to c
\end{equation}
for some sequence $\mu_k \in \R$. The first limit is equivalent to
\begin{equation} \label{3.6}
\int_{\R^N} \nabla u_k \cdot \nabla v + V(x)\, u_k\, v = c_k \int_{\R^{N-1}} |u_k|^{p-2}\, u_k\, v + \o(\norm{v}) \quad \forall v \in H^1(\R^N)
\end{equation}
with $c_k = (p/2)\, \mu_k$. Since $\sqrt{J(\cdot)}$ is an equivalent norm on $H^1(\R^N)$, $u_k$ is bounded by the second limit in \eqref{3.5}, so taking $v = u_k$ shows that $c_k \to c$. If $u_k(\cdot + y_k) \rightharpoonup w$ on a renumbered subsequence for some $y_k \in \R^{N-1}$ with $|y_k| \to \infty$, replacing $v$ with $v(\cdot - y_k)$ in \eqref{3.6}, making the change of variable $x \mapsto x + y_k$, and passing to the limit using \eqref{1.12} now gives
\begin{equation} \label{3.7}
\int_{\R^N} \nabla w \cdot \nabla v + V^\infty\, wv = c \int_{\R^{N-1}} |w|^{p-2}\, wv \quad \forall v \in H^1(\R^N).
\end{equation}
Taking $v = w$ gives $\norm{w}^2 = c \norm[p]{w}^p$, where $\norm[p]{\cdot}$ denotes the $L^p$-norm in $\R^{N-1}$, so if $w \ne 0$, then it follows that $c > 0$ and $\norm{w} \ge \big((\lambda_1^\infty)^{p/2}/c\big)^{1/(p-2)}$ since $\norm{w}^2/\norm[p]{w}^2 \ge \lambda_1^\infty$.

\begin{lemma} \label{Lemma 3.2}
Let $u_k \in \M$ be a critical sequence for $\restr{J}{\M}$ at the level $c \in \R$. Then it admits a renumbered subsequence that satisfies, in addition to the conclusions of Proposition \ref{Proposition 3.1},
\begin{gather}
\label{3.8} \int_{\R^N} \nabla w^{(1)} \cdot \nabla v + V(x)\, w^{(1)}\, v = c \int_{\R^{N-1}} |w^{(1)}|^{p-2}\, w^{(1)}\, v \quad \forall v \in H^1(\R^N),\\[5pt]
\label{3.9} \int_{\R^N} \nabla w^{(n)} \cdot \nabla v + V^\infty\, w^{(n)}\, v = c \int_{\R^{N-1}} |w^{(n)}|^{p-2}\, w^{(n)}\, v \quad \forall v \in H^1(\R^N),\, n = 2,\dots,m,\\[7.5pt]
\label{3.10} J(w^{(1)}) = c\, I(w^{(1)}), \qquad J^\infty(w^{(n)}) = c\, I(w^{(n)}), \quad n = 2,\dots,m,\\[7.5pt]
\label{3.11} \sum_{n=1}^m\, I(w^{(n)}) = 1, \qquad J(w^{(1)}) + \sum_{n=2}^m\, J^\infty(w^{(n)}) = c,\\[5pt]
\label{3.12} u_k - \sum_{n=1}^m\, w^{(n)}(\cdot - y^{(n)}_k) \to 0 \text{ in } H^1(\R^N).
\end{gather}
\end{lemma}

\begin{proof}
Since $y^{(1)}_k = 0$, $u_k \rightharpoonup w^{(1)}$ by \eqref{3.1}, so \eqref{3.8} follows from \eqref{3.6}. For $n = 2,\dots,m$, \linebreak $u_k(\cdot + y^{(n)}_k) \rightharpoonup w^{(n)}$, and taking $l = 1$ in \eqref{3.2} shows that $\big|y^{(n)}_k\big| \to \infty$, so \eqref{3.9} follows from \eqref{3.7}. Taking $v = w^{(1)}$ in \eqref{3.8} gives the first equation in \eqref{3.10}, and taking $v = w^{(n)}$ in \eqref{3.9} gives the second. \eqref{3.12} follows from \eqref{3.4}, \eqref{3.6}, and the continuity of the Sobolev imbedding. The second equation in \eqref{3.11} follows from \eqref{3.10} and the first, so it only remains to prove that
\[
\sum_{n=1}^m\, \norm[p]{w^{(n)}}^p = 1.
\]

Since $\norm[p]{u_k} = 1$,
\begin{equation} \label{3.13}
\norm[p]{\sum_{n=1}^m\, w^{(n)}(\cdot - y^{(n)}_k)} \to 1
\end{equation}
by \eqref{3.4}. Let $\varepsilon > 0$. Since $C^\infty_0(\R^N)$ is dense in $H^1(\R^N) \hookrightarrow L^p(\R^{N-1})$, for $n = 1,\dots,m$, there is a $\widetilde{w}^{(n)} \in C^\infty_0(\R^N)$ such that $\norm[p]{\widetilde{w}^{(n)} - w^{(n)}} < \varepsilon/m$. Then
\begin{multline} \label{3.14}
\abs{\norm[p]{\sum_{n=1}^m\, \widetilde{w}^{(n)}(\cdot - y^{(n)}_k)} - \norm[p]{\sum_{n=1}^m\, w^{(n)}(\cdot - y^{(n)}_k)}} \le \sum_{n=1}^m\, \norm[p]{\widetilde{w}^{(n)}(\cdot - y^{(n)}_k) - w^{(n)}(\cdot - y^{(n)}_k)}\\[10pt]
= \sum_{n=1}^m\, \norm[p]{\widetilde{w}^{(n)} - w^{(n)}} < \varepsilon \quad \forall k.
\end{multline}
For sufficiently large $k$, the supports of $\widetilde{w}^{(n)}(\cdot - y^{(n)}_k)$ are pairwise disjoint by \eqref{3.2} and hence
\begin{equation} \label{3.15}
\norm[p]{\sum_{n=1}^m\, \widetilde{w}^{(n)}(\cdot - y^{(n)}_k)}^p = \sum_{n=1}^m\, \norm[p]{\widetilde{w}^{(n)}(\cdot - y^{(n)}_k)}^p = \sum_{n=1}^m\, \norm[p]{\widetilde{w}^{(n)}}^p.
\end{equation}
Combining \eqref{3.13} -- \eqref{3.15} gives
\[
\abs{\left(\sum_{n=1}^m\, \norm[p]{\widetilde{w}^{(n)}}^p\right)^{1/p} - 1} \le \varepsilon.
\]
Let $\varepsilon \to 0$.
\end{proof}

\subsection{Proof of Theorem \ref{Theorem 1.4}}

Specializing to the case $V(x) \equiv V^\infty$ in Lemma \ref{Lemma 3.2} gives the following lemma.

\begin{lemma} \label{Lemma 3.3}
Let $u_k \in \M$ be a critical sequence for $\restr{J^\infty}{\M}$ at the level $c \in \R$. Then it admits a renumbered subsequence that satisfies, in addition to the conclusions of Proposition \ref{Proposition 3.1},
\begin{gather}
\label{3.16} \int_{\R^N} \nabla w^{(n)} \cdot \nabla v + V^\infty\, w^{(n)}\, v = c \int_{\R^{N-1}} |w^{(n)}|^{p-2}\, w^{(n)}\, v \quad \forall v \in H^1(\R^N),\, n = 1,\dots,m,\\[7.5pt]
\label{3.17} J^\infty(w^{(n)}) = c\, I(w^{(n)}), \quad n = 1,\dots,m,\\[7.5pt]
\label{3.18} \sum_{n=1}^m\, I(w^{(n)}) = 1, \qquad \sum_{n=1}^m\, J^\infty(w^{(n)}) = c,\\[5pt]
\label{3.19} u_k - \sum_{n=1}^m\, w^{(n)}(\cdot - y^{(n)}_k) \to 0 \text{ in } H^1(\R^N).
\end{gather}
\end{lemma}

Let $u_k \in \M$ be a critical sequence for $\restr{J^\infty}{\M}$ at the level $\lambda_1^\infty$ that satisfies the conclusions of Lemma \ref{Lemma 3.3}. Then
\begin{equation} \label{3.20}
\sum_{n=1}^m\, I(w^{(n)}) = 1, \qquad \sum_{n=1}^m\, J^\infty(w^{(n)}) = \lambda_1^\infty
\end{equation}
by \eqref{3.18}, and writing $I(w^{(n)}) = t_n$ and using $J^\infty(w^{(n)}) \ge \lambda_1^\infty\, t_n^{2/p}$ gives
\[
\sum_{n=1}^m\, t_n = 1, \qquad \sum_{n=1}^m\, t_n^{2/p} \le 1.
\]
Since $p > 2$, this implies that there is exactly one nonzero $t_n$, say, $t_{n_0}$. Then it follows from \eqref{3.20} that $w^{(n_0)}$ is a minimizer.

Let $w_1^\infty = \abs{w^{(n_0)}}$ and note that $w_1^\infty \ge 0$ is also a minimizer. If $w_1^\infty(x_0) = \nolinebreak 0$ for some \linebreak $x_0 \in \R^N \setminus \R^{N-1}$, then by the strong maximum principle, $w_1^\infty = 0$ in the half-space $\Omega$ determined by $\R^{N-1}$ that contains $x_0$, and hence also on $\R^{N-1}$ by the continuity of the trace imbedding $H^1(\Omega) \hookrightarrow L^p(\R^{N-1})$. This is impossible since $\norm[p]{w_1^\infty} = 1$, so $w_1^\infty > 0$ in $\R^N \setminus \R^{N-1}$. If $w_1^\infty(x_0) = 0$ at some $x_0 \in \R^{N-1}$, then by the Hopf lemma, $\partial (w_1^\infty)_+/\partial x_N(x_0) > 0$ and $\partial (w_1^\infty)_-/\partial x_N(x_0) < 0$, so $\jump{\partial w_1^\infty/\partial x_N}(x_0) > 0$. This violates the jump condition at $x_0$, so $w_1^\infty > 0$ on $\R^{N-1}$ as well.

\subsection{Proof of Theorem \ref{Theorem 1.5}}

Let $u_k \in \M$ be a critical sequence for $\restr{J}{\M}$ at the level $\lambda_1$ that satisfies the conclusions of Lemma \ref{Lemma 3.2}. Then
\begin{equation} \label{3.21}
\sum_{n=1}^m\, I(w^{(n)}) = 1, \qquad J(w^{(1)}) + \sum_{n=2}^m\, J^\infty(w^{(n)}) = \lambda_1
\end{equation}
by \eqref{3.11}, and writing $I(w^{(n)}) = t_n$ and using $J(w^{(1)}) \ge \lambda_1\, t_1^{2/p},\, J^\infty(w^{(n)}) \ge \lambda_1^\infty\, t_n^{2/p}$ for $n = 2,\dots,m$ gives
\[
\sum_{n=1}^m\, t_n = 1, \qquad \lambda_1\, t_1^{2/p} + \lambda_1^\infty \sum_{n=2}^m\, t_n^{2/p} \le \lambda_1.
\]
Since $p > 2$ and $\lambda_1 < \lambda_1^\infty$, this implies that $m = 1$. Then it follows from \eqref{3.21} that $w_1 = \abs{w^{(1)}}$ is a minimizer, and $w_1 > 0$ as in the proof of Theorem \ref{Theorem 1.4}.

\def\cdprime{$''$}

\end{document}